%% file: main.tex
\documentclass{amsart}
\usepackage{amssymb}
\usepackage[shortlabels]{enumitem}
\usepackage[utf8]{inputenc}
\usepackage{comment}
\usepackage{tikz-cd}
\usepackage{hyperref}

\newtheorem{theorem}{Theorem}[section]
\newtheorem{thmx}{Theorem}

\newtheorem{corx}[thmx]{Corollary}
\newtheorem{proposition}[theorem]{Proposition}
\newtheorem{lemma}[theorem]{Lemma}
\newtheorem{corollary}[theorem]{Corollary}

\theoremstyle{definition}
\newtheorem{definition}[theorem]{Definition}

\theoremstyle{remark}
\newtheorem{example}[theorem]{Example}				
\newtheorem{remark}[theorem]{Remark}

\newcommand{\Q}{\mathbb{Q}}

\newcommand{\Z}{\mathbb{Z}}

\newcommand{\E}{\mathcal{E}}

\def\hght{\mathrm{ht}}
\def\aut{\mathrm{aut}}

\def\top{\mathcal{T}op}
\def\asc{\mathcal{A}\mathcal{S}\mathcal{C}}
\def\finite{\mathcal{F}in}

\def\RG{\operatorname{R_G}}
\def\sd{\operatorname{sd}}

\if[]
\fi
\title[Realization of finite group actions]{Rigidifying simplicial complexes and realizing group actions}

\author{Cristina Costoya}
\address{CITMAga, Departamento de Matem{á}ticas, Universidade de Santiago de Compostela, 15705-A Coru{ñ}a, Spain.}
\email{cristina.costoya@usc.es}

\author{Rafael Gomes}
\address{Departamento de
Algebra, Geometr{í}a y Topolog{í}a, Universidad de M{á}laga, Campus de Teatinos, s/n, 29071-M{á}laga, Spain}
\email{rmgomes@uma.es}

\author{Antonio Viruel}
\address{Departamento de
Algebra, Geometr{í}a y Topolog{í}a, Universidad de M{á}laga, Campus de Teatinos, s/n, 29071-M{á}laga, Spain}
\email{viruel@uma.es}

\date{\today}

\thanks{The second author was supported by FCT - Funda{ç}{ã}o para a Ci{ê}ncia e Tecnologia (Portugal) with reference project 2021.04682.BD and DOI identifier 10.54499/2021.04682.BD. All authors were supported by MECIN/AEI/10.13039/501100011033 Grant PID2023-149804NB-I00.} % Funding footnote


\begin{document}

\begin{abstract}
We show that any action of a finite group on a finitely presentable group arises as the action of the group of self-homotopy equivalences of a space on its fundamental group. In doing so, we prove that any finite connected (abstract) simplicial complex $\mathbf{K}$ can be rigidified -- meaning it can be perturbed in a way that reduces the full automorphism group to any subgroup -- while preserving the homotopy type of the geometric realization $\vert \mathbf{K} \vert$. We also obtain that every action of a finite group on a finitely generated abelian group is the action of the group of self-homotopy equivalences of a space on one of its higher homotopy groups.
\end{abstract}

\maketitle

\section{Introduction} \label{section:intro}
In algebraic topology, one is often interested in the interplay between algebra and topology. This makes algebraic topology a field in which realizability questions emerge very naturally. These questions typically ask to what extent a given algebraic structure arises as a topological or homotopical invariant. In the 1960's, Steenrod asked which graded algebras appear as the cohomology of a space. This problem has been extensively studied \cite{aguade1982} and settled in the case of polynomial algebras \cite{andersen2008steenrod}. Steenrod also proposed the \textit{$G$-Moore space problem}, $G$ a group, that seeks a characterization of the $\Z G$-modules that appear as the homology of some simply-connected $G$-Moore space \cite[Problem 51]{lashof}. Carlsson proved that not all $\Z G$-modules can be realized \cite{carlsson} and Vogel characterized the finite groups $G$ for which all $\Z G$-modules are realizable \cite{Vogel1987ASO}. On the other hand, D.\ Kahn formulated the so-called \textit{realizability problem for abstract groups}, that given a group $G$, asks if there is a simply-connected space whose group of self-homotopy equivalences is isomorphic to $G$ \cite{kahn}. The first and third authors gave a positive answer to the latter problem when $G$ is a finite group using rational homotopy theory techniques in \cite{CV:finitegroupis} and for a larger class of groups in \cite{CV:actions}, where any faithful action of a finite group on a $\Q$-module is realized as the action of the group of self-homotopy equivalences of a space on a certain higher homotopy group. 

Both the \textit{$G$-Moore space problem} and the \textit{realizability problem for abstract groups} focus on realizing algebraic structures as homotopy invariants of simply-connected spaces. In this article, we examine a complementary setting: we seek to realize an action of a group $G$ on a group $M$ as the action of the group of self-homotopy equivalences of a space $X$ on the fundamental group of $X$. This perspective provides a natural complement to \cite[Question 1.3]{CV:actions}, which asks when the action of a group $G$ on an abelian group $M$ arises as the action of the group of self-homotopy equivalences $\E(X)$ of a space $X$ on a higher homotopy group $\pi_k(X)$ of $X$, for some $k \geq 2$.

We achieve this realization when $G$ is a finite group acting, not necessarily faithfully, on a finitely presentable group $M$. Our main result, which appears as Theorem \ref{main:theorem:text} in Section \ref{section:finite:spaces}, is as follows.
\begin{thmx}\label{main:theorem}
 Let $G$ be a finite group acting on a finitely presentable group $M$. Then there is a topological space $X$ such that:
    \begin{enumerate}[label={\rm (\roman{*})}] 
        \item\label{main:theorem:1} $\E(X) \cong G$;
        \item\label{main:theorem:2}  $\pi_1(X) \cong M$;
        \item\label{main:theorem:3} the action $G\curvearrowright M$ is equivalent to the canonical action $\E(X) \curvearrowright \pi_1(X)$.
    \end{enumerate}
\end{thmx}
A noteworthy feature of Theorem \ref{main:theorem} is that, unlike our previous work \cite{cgv:perm} and several other works by the first and third authors \cite{CV:actions}, \cite{CV:finitegroupis} and Méndez \cite{cmv:arrow:cat}, \cite{cmv:arrow:coalg}, we do not rely on a result by Frucht \cite{Frucht}, which says that any finite group is the full automorphism group of a graph. Instead, in order to realize the group $G$ as the group of self-homotopy equivalences of a space, we rely on the combinatorial and geometric nature of simplicial complexes and the theory of (minimal) finite spaces. Furthermore, Theorem \ref{main:theorem} is in a certain sense a generalization of our previous work \cite{cgv:perm}, since, for $M$ abelian, Hurewicz's Theorem implies that we also realize the action the group of self-homotopy equivalences on the first homology group. In \cite{cgv:perm}, we specifically realized a permutation module as the homology of a space. 

Furthermore, we obtain, as a consequence of Theorem~\ref{main:theorem}, that every action of a finite group $G$ on a finitely generated abelian group $M$ arises as the canonical action of the group of self-homotopy equivalences $\E(X)$ of a space $X$ on $\pi_k(X)$, for some $k \geq 2$. This completely addresses \cite[Question 1.3]{CV:actions} in the case $M$ is finitely generated and appears as Corollary~\ref{main:cor:text}.

\begin{corx}\label{main:cor}
    Let $k \geq 2$ be a natural number, $G$ a finite group and $M$ a finitely generated $\Z G$-module. Then, there is a topological space $X$ such that:
        \begin{enumerate}[label={\rm (\roman{*})}]
            \item\label{main:cor:1} $\E(X) \cong G$;
            \item\label{main:cor:2} $\pi_k(X) \cong M$;
            \item\label{main:cor:3} the action $G\curvearrowright M$ is equivalent to the canonical action $\E(X) \curvearrowright \pi_k(X)$.
        \end{enumerate}
\end{corx}

The strategy to prove Theorem \ref{main:theorem} can be summarized as follows. Firstly, we show that $M$ admits a $G$-symmetric presentation: a presentation on which $G$ acts by permuting generators and relations. This allows us to construct an (abstract) simplicial complex $\mathbf{K}$ for which the fundamental group of the geometric realization $\vert \mathbf{K} \vert$ is $M$ and such that $G < \aut(\mathbf{K})$ acts via automorphisms. We then modify $\mathbf{K}$ to eliminate unnecessary automorphisms while preserving its homotopy type, ensuring that the action of $G \cong \aut(\mathbf{K})$ on $\pi_1(\vert \mathbf{K} \vert) \cong M$ is precisely the given action $G \curvearrowright M$. Finally, we approximate $\mathbf{K}$ by a minimal finite space $X$ so that $\aut(X) \cong \mathcal{E}(X)$ and realize $G \curvearrowright M$ as $\mathcal{E}(X) \curvearrowright \pi_1(X)$. From this strategy resulted a theorem of independent interest concerning the full automorphism group of finite (connected) simplicial complexes. We state it below and it can be found as Theorem \ref{simplicial:complex:thm} in Section \ref{section:simp:comp}.  
\begin{thmx} \label{secondary:theorem}
Let $\mathbf{K}$ be a finite connected (abstract) simplicial complex and let $G$ be a subgroup of the full automorphism group $\aut(\mathbf{K})$ of $\mathbf{K}$. Then, there is a simplicial complex $\mathbf{L}$ such that:
\begin{enumerate}[label={\rm (\roman{*})}]
    \item $\aut(\mathbf{L}) \cong G$;
    \item $\vert \mathbf{L} \vert$ is homotopy equivalent to $\vert \mathbf{K} \vert$.
\end{enumerate}
\end{thmx}

In order to prove Theorem \ref{secondary:theorem}, we will introduce the notion of \textit{$G$-rigidification of a finite connected (abstract) simplicial complex} $\mathbf{K}$ (with respect to the path $P$), denoted $\RG(\mathbf{K};P)$, where $G$ is a subgroup of $\aut(\mathbf{K})$ and $P$ is a path containing all the points of $\mathbf{K}$. The simplicial complex $\RG(\mathbf{K};P)$ is obtained by gluing contractible simplicial complexes to $\mathbf{K}$ through the path $P$ and its orbits $g \cdot P$ allowing us to reduce the automorphism group to $G$ without changing the homotopy type. When $G$ is the trivial group $\{1\}$, the $G$-rigidification does not yield a simplicial complex with the desired properties. In this case, we glue (different) contractible simplicial complexes in each point, so that all non-trivial automorphisms are eliminated. 

\begin{remark}
The core idea behind Theorem~\ref{secondary:theorem} is the following. Given a permutation group $G$ acting on a finite set $A=\{a_1 \dots,a_n\}$, the group $G$  can be determined by the orbit of the $n$-tuple $(a_1, \dots,a_n)$ under the diagonal action. This means that two subgroups $G$ and $H$ of $\aut(A)$ coincide if and only if the $G$-orbit and $H$-orbit of $(a_1, \dots,a_n)$ are the same. This appears as Lemma~\ref{lemma:g:orbits}.

In our approach to Theorem~\ref{secondary:theorem}, we consider a finite connected (abstract) simplicial complex $\mathbf{K}$ with the action of its full automorphism group. We assume $\mathbf{K}$ to be connected so that there is a path containing all the vertices of $\mathbf{K}$ (with possible repetitions) and regard paths as $n$-tuples. Then, we modify the simplicial complex $\mathbf{K}$ along a path containing all points of $\mathbf{K}$ and its $G$-orbits, for $G$ a subgroup of $\aut(\mathbf{K})$. 
\end{remark}

\subsection*{Notation}
We use boldface symbols such as $\mathbf{K}$ and $\mathbf{L}$ to denote (abstract) simplicial complexes. The standard (abstract) simplicial complex over the set $A$ is denoted by $\mathbf{\Delta}(A)$ and $\vert \mathbf{K} \vert$ denotes the geometric realization of the (abstract) simplicial complex $\mathbf{K}$. Roman capital letters can denote several different things: $G$ is reserved for a finite group, $M$ for a finitely presentable group and $X$ is always a topological space. The letters $P$ and $Q$ will denote paths of an (abstract) simplicial complex. Calligraphic expressions denote categories: $\top$, $\asc$ and $\finite$ are respectively the categories of topological spaces, abstract simplicial complexes and finite spaces.

%In Section \ref{section:covers}, we exploit known results concerning the realizability of a finitely presented group as the fundamental group of a space to construct a simplicial complex $K$ for which the fundamental group of its geometric realization is $M$ and that realizes the action of $G$ on $M$ as the action of a subgroup of $\aut(K)$ on the fundamental group of $K$
\subsection*{Organization of the paper}
Section \ref{section:covers} is divided in two parts. In the first one, we show a finitely presentable group $M$ with a $G$-action admits a $G$-symmetric presentation and use this to construct a simplicial complex whose fundamental group is $M$ and on which $G$ acts simplicially. In the second one, we introduce auxiliary simplicial complexes that play a key role in subsequent sections. In Section \ref{section:simp:comp}, we construct the \textit{$G$-rigidification of a simplicial complex} $\mathbf{K}$, denoted $\RG(\mathbf{K})$, and show this is a simplicial complex which is homotopy equivalent to $\mathbf{K}$ and whose group of automorphisms $\aut(\RG(\mathbf{K}))$ is $G$. Finally, in Section \ref{section:finite:spaces}, by considering the poset of simplices of $\RG(\mathbf{K})$ and eliminating beat points in a careful manner, we construct a finite space $X$ that is minimal, so that $\aut(X) \cong \E(X)$, and inherits the realizability properties of $\RG(\mathbf{K})$.

\section{Foundations} \label{section:covers}
\subsection{$G$-symmetric presentations} %just to separate from what is already written
In this subsection, $G$ denotes a finite group acting on a finitely presentable group $M$ and $F_n$ is the free group of rank $n$. The aim of the current subsection is to construct a simplicial complex on which the group $G$ acts by automorphisms and whose geometric realization has fundamental group isomorphic to $M$. In order to do that, we need a presentation of $M$ compatible with the $G$-action.

\input{symmetric_presentations}

Our aim in Section \ref{section:simp:comp} is to eliminate some of the automorphisms of $W$ in such a way that the fundamental group $\pi_1(W)$ does not change. Since it is easier to work with automorphisms of (abstract) simplicial complexes, we will work in that setting. No issues arise from such transition. Indeed, we can assume $W$ is a $\Delta$-complex, as the following lemma shows. We follow the definition of $\Delta$-complex appearing in \cite[Section 2.1]{hatcher} and recall that $\Delta$-complexes were introduced by Eilenberg and Zilber in \cite{eilenber:zilber:semi} under the name \textit{semi-simplicial complexes}.

\begin{lemma} \label{lemma:Delta:cx}
Given a finite group $G$ acting on a finitely presentable group $M$, there is a $\Delta$-complex $Y$ such that:
\begin{enumerate}[(i)]
    \item $\pi_1(Y)\cong M$;
    \item $G$ acts on $Y$ by homeomorphisms;
    \item $G \curvearrowright M$ is equivalent to $G \curvearrowright \pi_1(Y)$.
\end{enumerate}
Furthermore, the action of $G$ on $Y$ takes $n$-simplices to $n$-simplices by the canonical linear homeomorphism preserving orientation.
\end{lemma}
\begin{proof}
    Assume $M$ has a $G$-symmetric presentation
    $$M \cong \langle x_1,\ldots, x_n\, |\, w_j(x_1,\ldots, x_n)=1, 1\leq j\leq r\rangle$$
    with permutation representations given by $\xi:G \rightarrow \Sigma_n$ and $\rho:G \rightarrow \Sigma_r$. Suppose, additionally, that the letters of the words $w_j$ only have exponents $\pm 1$ and let $l_j$ denote the length of the word $w_j$. 
    
    Consider the $l_j$-polygon (with interior) $P_j$ whose edges are the letters of the word $w_j$ oriented according to the exponents that appear in $w_j$. If $l_j=2 $, then $P_j$ has two vertices and two edges between them. In this case, $P_j$ is not a polygon in the usual sense, but no issues arise from it. Now, we triangulate each of the polygons $P_j$ by adding a vertex, the barycenter of $P_j$, and edges between the vertices and the barycenter (oriented towards the barycenter). Let $S(P_j)$ denote the set of triangles of $P_j$ (and think of them as disjoint oriented topological $2$-simplices). Then, $Y$ is the following quotient
    $$Y= \left(\bigsqcup_{j=1}^r S(P_j) \sqcup \bigsqcup_{\text{free generators}}m_{k_i} \right) \big/ \sim$$
    where the second union runs over all free generators - those that do not appear in any relation; $m_{k_i}$ denotes both the generator and a loop - a $1$-simplex starting and ending at the same point; and $\sim$ identifies:
    \begin{itemize}
        \item the boundary of the $2$-simplices of the triangulation of each polygon $P_j$ in the obvious way that recovers the polygons;
        \item all the vertices (ends of edges of the boundary) of the polygons $P_j$ and the vertices of the loops $m_{k_i}$ to a single point;
        \item the edges labeled with the same letter.
    \end{itemize} 
    All the identifications of subsimplices are done according to the canonical linear homeomorphism that preserves orientation, hence $Y$ is a $\Delta$-complex. See \cite[Section 2.1]{hatcher} for more details on this characterization of $\Delta$-complexes.

    Clearly, $\pi_1(Y) \cong M$, since without the triangulation $Y$ is just a CW-complex with a single $0$-cell, a $1$-cell for each generator and a $2$-cell for each relation, whose attaching map is prescribed by the relation. On the other hand, $Y$ comes equipped with a $G$-action that permutes labeled edges and polygons. This action has a fixed point: the only point that belongs to every simplex of $Y$, so it induces an action on $\pi_1(Y)\cong M$, which is equivalent to $G \curvearrowright M$.
\end{proof}
By \cite[Chapter 2; Exercise 23]{hatcher}, if we apply barycentric subdivision twice to a $\Delta$-complex of Lemma \ref{lemma:Delta:cx}, we reduce to the desired setting of simplicial complexes. We summarize this in the following theorem.
\begin{theorem} \label{simp:comp:covering}
    Let $G$ be a finite group acting on a finitely presentable group $M$. Then there is an (abstract) simplicial complex $\mathbf{K}$ such that: 
    \begin{enumerate}[label={\rm (\roman{*})}]
        \item\label{simp:comp:covering:1} $\pi_1(\vert \mathbf{K} \vert )\cong M$;
        \item\label{simp:comp:covering:2} $G$ acts simplicially on $\vert \mathbf{K} \vert$, hence $G \leq \aut(\mathbf{K})$;
        \item\label{simp:comp:covering:3} $G \curvearrowright M$ is equivalent to $ G \curvearrowright \pi_1(\vert \mathbf{K} \vert)$.
    \end{enumerate}
\end{theorem}
Despite the similarities between Theorem~\ref{simp:comp:covering} and Theorem~\ref{main:theorem}, the latter remains far from being established. In the following example, we show there is a finitely presentable group $M$, on which $G=\Z/2$ acts, and an (abstract) simplicial complex $\mathbf{K}$ satisfying all the properties of Theorem~\ref{simp:comp:covering} while $\E(\vert \mathbf{K} \vert) \neq \Z/2$.
\begin{example}
Fix $n \geq 2$ and let $H$ be the group given by the following presentation
$$\langle a, t \mid t^{-1}a^{-1}ta^nt^{-1}at = a \rangle.$$
A larger class of groups that includes $H$ was studied by A. M. Brunner in \cite{brunner}, where it is shown that the full automorphism group of $H$ is not finitely generated. Furthermore, A.M. Brunner and J. G. Ratcliffe prove in \cite{brunner:ratcliffe} that the space with a single $0$-cell, two $1$-cells and a single $2$-cell attached by the relation in $H$ is an Eilenberg-Maclane space $K(H,1)$.

Denote the free product of $H$ with itself by $M=H \ast H$ and endow $M$ with the $G=\Z/2$-action that switches generators and relations between the two copies of $H$ within $M$. Then, Theorem~\ref{simp:comp:covering} applied to this action gives a simplicial complex $\mathbf{K}$ whose geometric realization is homotopy equivalent to $K(H,1) \vee K(H,1)$. Since $\aut(H) \cong \E(K(H,1)) < \E(\vert \mathbf{K}\vert)$, we conclude that $\E(\vert \mathbf{K}\vert)$ is not finitely generated while $G=\Z/2$ is the smallest non-trivial finite group.
\end{example}
% OLD Rigidification of simplicial complexes STARTED HERE
\subsection{Auxiliary simplicial complexes}

For the convenience of the reader, we recall the notation about simplicial complexes. Calligraphic letters such as $\mathbf{K}$ denote an (abstract) simplicial complex and $V(\mathbf{K})$ its vertex set. The geometric realization functor from the category of (abstract) simplicial complexes to topological spaces is denoted by $\vert \cdot \vert: \asc \longrightarrow \top$ while $\mathbf \Delta(A)$ is the standard simplex whose vertices are the elements of $A$.  Whenever we define a simplicial complex by its set of simplices, it is implied that the vertex set consists of all the elements that appear in the simplices (regarded as sets). 

The main objective of Section \ref{section:simp:comp} is perturbing a simplicial complex $\mathbf{K}$ from Theorem \ref{simp:comp:covering} in such a way that the full automorphism group of the new simplicial complex is the subgroup $G$ of $\aut (\mathbf{K})$ while the homotopy type of the geometric realization $\vert \mathbf{K} \vert$ remains unchanged. The rest of the current section is devoted to recording auxiliary lemmas about properties of automorphisms of simplicial complexes and defining specific simplicial complexes that will be useful to us in the aforementioned process.

\begin{lemma} \label{points:auto:simplices}
    Let $f$ be an automorphism of $\mathbf{K}$. If $x \in V(\mathbf{K})$ is exactly in $k$ $n$-simplices, then $f(x) \in V(\mathbf{K})$ is exactly in $k$ $n$-simplices for all $k \geq 0$ and $n \geq 1$.
\end{lemma}
\begin{proof}
    Since $f$ is an automorphism of $\mathbf{K}$, it is a map from the vertex set $V(\mathbf{K})$ to itself that preserves simplices in the sense that if $\{v_0, \dots, v_n \}$ is a simplex of $\mathbf{K}$, then $\{f(v_0), \dots, f(v_n) \}$ is a simplex (of the same dimension) of $\mathbf{K}$. This implies directly that if $x \in V(\mathbf{K})$ is in exactly $k$ $n$-simplices, then $f(x)$ is in at least $k$ $n$-simplices. Applying the same reasoning to the automorphism $f^{-1}$, we get the result.
\end{proof}
We say that an $(n+1)$-tuple $(v_0, \dots, v_n) \in V(\mathbf{K})^{n+1}$ is a \textit{path of length $n$} of $\mathbf{K}$ if $\{v_i,v_{i+1}\} \in \mathbf{K}$ for all $0 \leq i \leq n-1$, that is, $\{v_i,v_{i+1}\}$ is a 1-simplex of $\mathbf{K}$ for all $0 \leq i \leq n-1$.

\begin{lemma}\label{paths:auto:simplices}
    Let $f$ be an automorphism of $\mathbf{K}$. Then $(v_0, \dots, v_n) \in V(\mathbf{K})^{n+1}$ is a path of length $n$ if and only if $(f(v_0), \dots, f(v_n)) \in V(\mathbf{K})^{n+1}$ is a path of length $n$ of $\mathbf{K}$
\end{lemma}
\begin{proof}
    This follows directly from the fact that $f$ and $f^{-1}$ preserve 1-simplices.
\end{proof}

We define $\mathbf{P_k}$ to be the simplicial complex whose set of vertices is $V(\mathbf{P_k})=\{u_0 \dots,u_k\}$ and whose simplices are given by 
$$\mathbf{P_k}= \bigcup_{i= 0}^{k-1}\{u_i,u_{i+1}\}.$$ 
This simplicial complex essentially corresponds to a single path of length $k$. With that in mind, $\mathbf{P_k}$ can be described as the following sequence of pushouts of the standard $1$-simplex:
$$ \mathbf{\Delta}(u_0,u_1) \bigsqcup_{u_1} \mathbf{\Delta}(u_1,u_2) \bigsqcup_{u_2}\dots \bigsqcup_{u_{k-1}}\mathbf{\Delta}(u_{k-1},u_k).$$
\begin{definition} \label{simp:comp:W_k,k+1}
    Let $\mathbf{W_k}$ be the simplicial complex given by the following pushout

     \begin{equation}
            \begin{tikzcd}
                \ast \bigsqcup \ast \arrow[r,"v_0 \sqcup v_1"] \arrow[d,"u_0 \sqcup u_0"] & \mathbf{\Delta}(\{v_0,v_1,v_2\}) \arrow[d] \\
                \mathbf{P_k} \bigsqcup \mathbf{P_{k+1}} \arrow[r] & \mathbf{W_k}  \arrow[ul, phantom, "\text{\Huge$\ulcorner$}", pos=0]
            \end{tikzcd} \nonumber
        \end{equation}
where the symbol $\bigsqcup$ denotes the disjoint union.
    More explicitly, after relabeling the vertices of $\mathbf{P_{k+1}}$, $\mathbf{W_k}$ is given by:
    $$\mathbf{W_k}= \mathbf\Delta(\{v_0,v_1,v_2\}) \bigcup \{v_0,u_1 \} \cup \bigcup_{i=1}^{k-1} \{u_i, u_{i+1} \} \bigcup \{v_1,t_1 \} \cup \bigcup_{i=1}^{k} \{t_i, t_{i+1} \}.$$
\end{definition}
In simpler terms, $\mathbf{W_k}$ consists of a standard $2$-simplex with an added path of length $k$ in one of the points and an added path of length $k+1$ in another of the points. The geometric simplicial complex corresponding to $\mathbf{W_2}$ can be found in Figure \ref{fig:rigid_contractible_simp_complex}.
\begin{lemma}
    Let $\mathbf{W_k}$ be the simplicial complex from Definition \upshape\ref{simp:comp:W_k,k+1}. Then the following holds:
    \begin{enumerate}[label={\rm (\roman{*})}]
        \item $\vert \mathbf{W_k} \vert$ is contractible;
        \item $\mathbf{W_k} $ is rigid, i.e., $\aut(\mathbf{W_k})=\{1\}$.
    \end{enumerate}
\end{lemma}
\begin{proof}
    The geometric realization $\vert \mathbf{W_k} \vert$ is contractible, since it consists of the pointed union (in different points) of a standard topological simplex with distinct intervals. 
    
    Every automorphism of $\mathbf{W_k}$ fixes $v_2$, because it is the only vertex in a $2$-simplex and in exactly $2$ edges. Using lemmas \ref{points:auto:simplices} and \ref{paths:auto:simplices}, we get that every automoprphism fixes the remaining points, since $(v_1,t_1, \dots, t_{k+1})$ is the only path of length $k+1$ starting at a point in a $2$-simplex and whose remaining points are not in a $2$-simplex.
\end{proof}
\begin{figure}[h]
\centering
\begin{tikzpicture}

        % Define points
        \coordinate (v0) at (0,0);
        \coordinate (v1) at (3,0);
        \coordinate (v2) at (1.5,2);
        
        \coordinate (u1) at (-1,0);
        \coordinate (u2) at (-2,0);
        
        \coordinate (t1) at (4,0);
        \coordinate (t2) at (5,0);
        \coordinate (t3) at (6,0);
        
        % Draw the filled triangle
        \filldraw[fill=blue!20, draw=black] (v0) -- (v1) -- (v2) -- cycle;
        
        % Draw edges
        \draw[thick] (v0) -- (v1);
        \draw[thick] (v1) -- (v2);
        \draw[thick] (v2) -- (v0);
        
        % Draw additional paths
        \draw[thick] (v0) -- (u1) -- (u2); % Path of length 2
        \draw[thick] (v1) -- (t1) -- (t2) -- (t3); % Path of length 3

        % Draw vertices
        \foreach \point in {v0,v1,v2,u1,u2,t1,t2,t3}
            \fill (\point) circle (2pt);

        % Label vertices
        \node[below]  at (v0) {$v_0$};
        \node[below] at (v1) {$v_1$};
        \node[above] at (v2) {$v_2$};

        \node[below] at (u1) {$u_1$};
        \node[below] at (u2) {$u_2$};

        \node[below] at (t1) {$t_1$};
        \node[below] at (t2) {$t_2$};
        \node[below] at (t3) {$t_3$};

    \end{tikzpicture}
\caption{Geometric simplicial complex $\mathbf{W_2}$}
        \label{fig:rigid_contractible_simp_complex}
\end{figure}
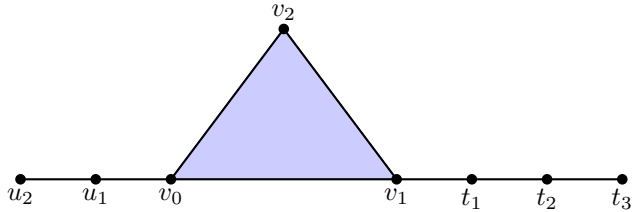

The simplicial complexes $\mathbf{W_k}$ from Definition \ref{simp:comp:W_k,k+1} will be used to construct a simplicial complex whose geometric realization is homotopy equivalent to $\vert \mathbf{K} \vert$ and whose (full) automorphism group is trivial. The simplicial complexes defined below will appear in the construction of the $G$-rigidification of a simplicial complex $\mathbf{K}$, that appears as Definition \ref{def:G:rigidif}.

\begin{definition}
    Define $\mathbf{U_{i}^n}$, for $n\geq 3$, to be the $n$-dimensional simplicial complex given by the following pushout

     \begin{equation}
            \begin{tikzcd}
                \bigsqcup_{j=0}^{n-3} \ast \arrow[r,"w_i \sqcup_j v_j"] \arrow[d,"\sqcup u_0"] & \mathbf \Delta(\{x_{i-1},x_i,y_i,w_i,v_1, \dots,v_{n-3}\}) \arrow[d] \\
                \bigsqcup_{j=1}^{n-3} \mathbf{P_{j}} \arrow[r] & \mathbf{U_{i}^n}  \arrow[ul, phantom, "\text{\Huge$\ulcorner$}", pos=0]
            \end{tikzcd} \nonumber
        \end{equation}
        
   %$$U_{k,i}^n= P_k \bigsqcup_{u_0 \sim w_i} \mathcal{P}(x_{i-1}, x_i,y_i,w_i, v_{1}, \dots,v_{n-3} \}) \bigsqcup_{\substack{v_j \sim u_0 \\j=1}}^{n-3} P_{k+j}.$$
   This simplicial complex consists of a standard $n$-simplex to which a path of different length is glued in all but three points: $x_{i-1}$, $x_i$ and $y_i$. Observe that $\mathbf{U_{i}^n}$ is not rigid, but every automorphism of $\mathbf{U_{i}^n}$ is determined by the images of $x_{i-1}$, $x_i$ and $y_i$. In Figure \ref{simp:comp:U_k^n}, we can find the corresponding geometric simplicial complex. 
   
   \begin{figure}[h]
        \centering
        \begin{tikzpicture}
    % Define coordinates for the base nodes
    \coordinate (X0) at (0,0);
    \coordinate (X1) at (4,0);

    % Define coordinates for the vertical towers (z values)
    \coordinate (Z11) at (2,4);

    % Define and label the y nodes
    \coordinate (Y1) at (2,1);

    % Define and label the w nodes
    \coordinate (W1) at (2,2);

    % Fill 3-simplices without drawing their boundaries
    \fill[blue!20] (X0) -- (Y1) -- (X1) -- cycle;
    \fill[blue!20] (X0) -- (Y1) -- (W1) -- cycle;
    \fill[blue!20] (X1) -- (Y1) -- (W1) -- cycle;

    % Draw solid edges (excluding the ones that should be dashed)
    \draw (X0) -- (X1);
    \draw (X0) -- (W1) -- (X1);
    \draw (W1) -- (Z11);
    
    % Draw dashed edges for connections involving y_1
    \draw[dashed] (X0) -- (Y1);
    \draw[dashed] (X1) -- (Y1);
    \draw[dashed] (Y1) -- (W1);

    % Draw the base nodes
    \filldraw [black] (X0) circle (2pt) node[below] {$x_0$};
    \filldraw [black] (X1) circle (2pt) node[below] {$x_1$};
    
    % Draw z nodes
    \filldraw [black] (Z11) circle (2pt) node[left] {$u_1$};
    \filldraw [black] (Y1) circle (2pt) node[above left] {$y_1$};
    \filldraw [black] (W1) circle (2pt) node[above left] {$w_1$};

\end{tikzpicture}
    \caption{Geometric realization of $\mathbf{U_1^3}$}
    \label{simp:comp:U_k^n}
\end{figure}
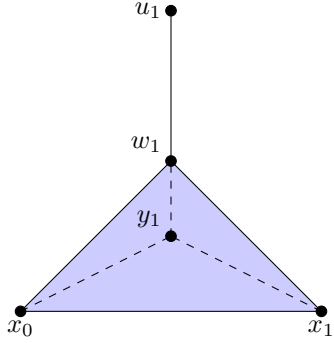 
 \label{def:aux:glued:band}
\end{definition}

\begin{definition} \label{def:glued:band}
Let $\mathbf{B^n_m}$ be the $n$-dimensional simplicial complex given by gluing the simplicial complexes $\mathbf{U^n_{i}}$ of Definition~\ref{def:aux:glued:band} for $1 \leq i \leq m$ subsequently to each other with $2$-simplices. More precisely, $\mathbf{B^n_m}$ is defined as follows
$$\mathbf{B^n_m}=  \left( \mathbf{U_{1}^{n+1}} \bigsqcup_{x_1}  \mathbf{U_{2}^{n+1}} \bigsqcup_{x_2} \dots  \bigsqcup_{x_m} \mathbf{U_{m}^{n+1}} \right) \cup \bigcup_{i=1}^{m-1} \left( \{x_i,w_i,x_{i+1}\} \cup  \{w_i,w_{i+1}\} \right).$$
In Figure \ref{glued:band}, we can find the geometric simplicial complex corresponding to $\mathbf{B_3^3}$, where the points in the paths added to each $\mathbf{U_{i}^3}$  were relabeled, so that they can be distinguished. 

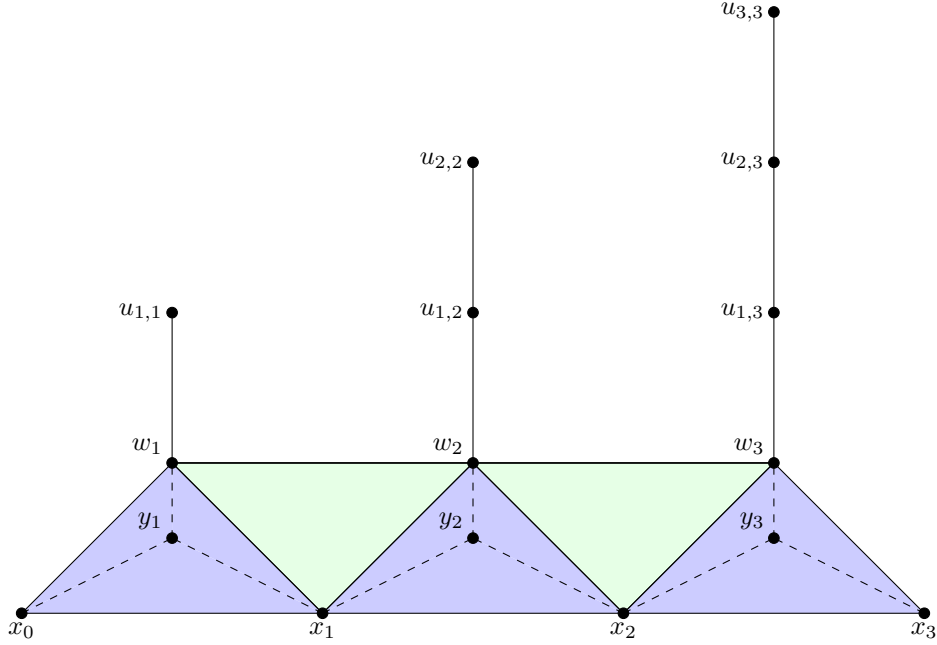
\begin{figure}[h]
        \centering
       \begin{tikzpicture}
    % Define coordinates for the base nodes
    \coordinate (X0) at (0,0);
    \coordinate (X1) at (4,0);
    \coordinate (X2) at (8,0);
    \coordinate (X3) at (12,0);

    % Define coordinates for the vertical towers (z values)
    \coordinate (Z11) at (2,4);
    \coordinate (Z12) at (6,4);
    \coordinate (Z22) at (6,6);
    \coordinate (Z13) at (10,4);
    \coordinate (Z23) at (10,6);
    \coordinate (Z33) at (10,8);

    % Define and label the y nodes
    \coordinate (Y1) at (2,1);
    \coordinate (Y2) at (6,1);
    \coordinate (Y3) at (10,1);

    % Define and label the w nodes
    \coordinate (W1) at (2,2);
    \coordinate (W2) at (6,2);
    \coordinate (W3) at (10,2);

    % Fill 3-simplices without drawing their boundaries
    \fill[blue!20] (X0) -- (Y1) -- (X1) -- cycle;
    \fill[blue!20] (X1) -- (Y2) -- (X2) -- cycle;
    \fill[blue!20] (X2) -- (Y3) -- (X3) -- cycle;
    \fill[blue!20] (X0) -- (Y1) -- (W1) -- cycle;
    \fill[blue!20] (X1) -- (Y1) -- (W1) -- cycle;
    \fill[blue!20] (X1) -- (Y2) -- (W2) -- cycle;
    \fill[blue!20] (X2) -- (Y2) -- (W2) -- cycle;
    \fill[blue!20] (X2) -- (Y3) -- (X3) -- cycle;
    \fill[blue!20] (X2) -- (Y3) -- (W3) -- cycle;
    \fill[blue!20] (X3) -- (Y3) -- (W3) -- cycle;

    % Draw dashed edges for connections involving y_i
    \draw[dashed] (X0) -- (Y1);
    \draw[dashed] (X1) -- (Y1);
    \draw[dashed] (Y1) -- (W1);

    \draw[dashed] (X1) -- (Y2);
    \draw[dashed] (X2) -- (Y2);
    \draw[dashed] (Y2) -- (W2);

    \draw[dashed] (X2) -- (Y3);
    \draw[dashed] (X3) -- (Y3);
    \draw[dashed] (Y3) -- (W3);

    % Draw 2-simplices
    \filldraw[fill=green!10, draw=black] (X1) -- (W1) -- (W2) -- cycle;
    \filldraw[fill=green!10, draw=black] (X2) -- (W2) -- (W3) -- cycle;

    % Draw the base nodes
    \filldraw [black] (X0) circle (2pt) node[below] {$x_0$};
    \filldraw [black] (X1) circle (2pt) node[below] {$x_1$};
    \filldraw [black] (X2) circle (2pt) node[below] {$x_2$};
    \filldraw [black] (X3) circle (2pt) node[below] {$x_3$};

    % Draw z nodes
    \filldraw [black] (Z11) circle (2pt) node[left] {$u_{1,1}$};
    \filldraw [black] (Z12) circle (2pt) node[left] {$u_{1,2}$};
    \filldraw [black] (Z13) circle (2pt) node[left] {$u_{1,3}$};
    \filldraw [black] (Z22) circle (2pt) node[left] {$u_{2,2}$};
    \filldraw [black] (Z23) circle (2pt) node[left] {$u_{2,3}$};
    \filldraw [black] (Z33) circle (2pt) node[left] {$u_{3,3}$};

    % Draw y nodes
    \filldraw [black] (Y1) circle (2pt) node[above left] {$y_1$};
    \filldraw [black] (Y2) circle (2pt) node[above left] {$y_2$};
    \filldraw [black] (Y3) circle (2pt) node[above left] {$y_3$};

    % Draw w nodes
    \filldraw [black] (W1) circle (2pt) node[above left] {$w_1$};
    \filldraw [black] (W2) circle (2pt) node[above left] {$w_2$};
    \filldraw [black] (W3) circle (2pt) node[above left] {$w_3$};

    % Draw vertical lines for z towers
    \draw (W1) -- (Z11);
    \draw (W2) -- (Z12) -- (Z22);
    \draw (W3) -- (Z13) -- (Z23) -- (Z33);

    % Draw edges connecting the base nodes and w nodes
    \draw (X0) -- (X1) -- (X2) -- (X3);
    \draw (X0) -- (W1) -- (X1) -- (W2) -- (X2) -- (W3) -- (X3);

    % Draw edges connecting w nodes to each other
    \draw (W1) -- (W2);
    \draw (W2) -- (W3);
    \draw (W1) -- (W3);
    
\end{tikzpicture}
    \caption{Geometric realization of $\mathbf{B^3_3}$}
    \label{glued:band}
\end{figure}

\end{definition}

\begin{lemma} \label{lemma:glued:band}
Let $\mathbf{B_m^{n}}$ be the simplicial complex from Definition \upshape\ref{def:glued:band}. Then the following hold:
    \begin{enumerate}[label={\rm (\roman{*})}]
        \item\label{lemma:glued:band:1} $\vert \mathbf{B_m^{n}} \vert$ is contractible.
        \item\label{lemma:glued:band:2} The automorphisms of $\mathbf{B_m^{n}} $ are determined by the images of $y_1$ and $y_m$. Moreover, if these are fixed by an automorphism $f$, then $f$ is the identity.
    \end{enumerate} 
\end{lemma}
\begin{proof}
    Point \ref{lemma:glued:band:1} is clear, since the geometric realization $\vert \mathbf{B_m^{n}} \vert$ consists of several wedge sums of a standard topological simplex with distinct paths. 
    
    In order to show point \ref{lemma:glued:band:2}, observe that $w_i$ is the only point satisfying the following properties simultaneously: 
    \begin{itemize}
        \item $w_i$ belongs to a $2$-simplex which is not a subsimplex of an $n$-simplex;
        \item $w_i$ starts a path of length $i$ of which the remaining points are not in any $2$-simplex;
    \end{itemize}
    where $1 \leq i \leq m $. Therefore, by Lemma \ref{points:auto:simplices}, each $w_i$ is fixed by an automorphism of $\mathbf{B_m^n}$, which implies that each $n$-simplex is mapped to itself. Consequently, the points $x_i$ with $2 \leq i \leq m-1$ are also fixed, as each $x_i$ is the unique point in the intersection of the $n$-simplex containing $w_i$ and the $n$-simplex containing $w_{i+1}.$
    
    Within each $n$-simplex, each $v_j$ for $1 \leq j \leq n-3$ starts a path of distinct length whose remaining points are not in any $2$-simplex, thus, by Lemma \ref{paths:auto:simplices}, each $v_j$ is also fixed by automorphisms of $\mathbf{B_m^n}$. Additionally, the points $y_i$ with $2 \leq i \leq m-1$ are fixed by automorphisms of $\mathbf{B_m^n}$, since $y_i$ is the only point in its corresponding $n$-simplex that does not belong to a simplex, which is not a subsimplex of that $n$-simplex. 
    
    Hence, the only points that may remain unfixed are $x_0,y_1,x_m$ and $y_m$. Again, since each $n$-simplex is mapped to itself, the image of $y_1$ is either $x_0$ or $y_1$ and the image of $y_m$ is either $x_m$ or $y_m$.
\end{proof}

\section{Rigidifying simplicial complexes} \label{section:simp:comp}
We are now ready to define the $G$-rigidification of a finite (connected) simplicial complex $\mathbf{K}$. The idea is the following: we fix a path $P= (x_0, \dots, x_m)$ of $\mathbf{K}$ containing all the points of $\mathbf{K}$ (with possible repetitions) and glue on $P$ a band $\mathbf{B_m^{n+1}}$ from Definition \ref{def:glued:band} by identifying the points $x_0, \dots, x_m$ of $\mathbf{B_m^{n+1}}$ with the points of $P$ with the same label. Then, for each $g \in G$, we glue a copy of the band $\mathbf{B_m^{n+1}}$ on the path $g \cdot P= (g \cdot x_0, \dots, g \cdot x_m)$ by identifying each $x_i \in \mathbf{B_m^{n+1}}$ with $g\cdot x_i \in g \cdot P$.

\begin{definition} \label{def:G:rigidif}
    Let $\mathbf{K}$ be a connected $n$-dimensional simplicial complex with non-trivial automorphism group $\aut(\mathbf{K})$ and let $G \leq \aut(\mathbf{K})$ be a subgroup. We define the \textit{$G$-rigidification of $\mathbf{K}$ with respect to the path $P$} as 
    $$\RG(\mathbf{K};P)= \mathbf{K} \cup \bigcup_{g \in G} \mathbf{B_g}$$
    where $P=(x_0, \dots, x_m)$ is a path containing all points of $\mathbf{K}$ (with possible repetitions) and $\mathbf{B_g}$ is the copy of $\mathbf{B_m^{n+1}}$ for which the points $x_0, \dots, x_m$ are (in order) the points of the path $g \cdot P$. We will denote the points of $\mathbf{B_g} \setminus \mathbf{K}$ with the superscript $g$.
    
Whenever it is clear which path we are considering, we will drop the $P$ from the notation of $\RG(\mathbf{K};P)$ and write it simply as $\RG(\mathbf{K})$. In such a case, we call it the \textit{$G$-rigidification of $\mathbf{K}$}.
\end{definition} 

\begin{lemma} \label{lemma:g:orbits}
    Let $G $ be a subgroup of the symmetric group $S_n=\aut(\{x_1, \dots,x_n\})$. Then the orbit of the $n$-tuple $(x_1, \dots,x_n)$ under the diagonal action of $G$ uniquely determines $G$: if another subgroup $K \leq S_n$ is such that 
    $$K \cdot (x_1, \dots,x_n)=G\cdot (x_1, \dots,x_n),$$ then $G=K$.
\end{lemma}
\begin{proof}
    We show that if $G \neq K$, then the $G$-orbit of $(x_1, \dots,x_n)$ does not coincide with the $K$-orbit of $(x_1, \dots,x_n)$. Let $f \in K \setminus  G$. Then, for each $g \in G$ there is an element $x_{i_g} \in \{x_1 \dots,x_n\}$ such that $f(x_{i,g})\neq g(x_{i,g})$. Hence, $f \cdot(x_1, \dots,x_n)=(f(x_1), \dots, f(x_n)) \notin G \cdot (x_1, \dots, x_n)$ and $K \cdot (x_1, \dots,x_n) \neq G\cdot (x_1, \dots,x_n)$.
\end{proof}
\begin{proposition} \label{prop:of:rigidif}
Let $\mathbf{K}$ be a finite connected $n$-dimensional simplicial complex with non-trivial automorphism group $\aut(\mathbf{K})$. Let $G \leq \aut(\mathbf{K})$ be a non-trivial subgroup and $P$ a path containing all the points of $\mathbf{K}$ (with possible repetitions). Then, $\RG(\mathbf{K};P)$, the \textit{$G$-rigidification of $\mathbf{K}$ with respect to the path $P$}, satisfies the following:
\begin{enumerate}[label={\rm (\roman{*})}]
    \item\label{prop:of:rigidif:1} $\aut(\RG(\mathbf{K};P)) \cong G \leq \aut(\mathbf{K})$; 
    \item\label{prop:of:rigidif:2} $\vert \RG(\mathbf{K};P) \vert  \simeq \vert \mathbf{K} \vert$. 
\end{enumerate}
\end{proposition}
\begin{proof}
In this proof, we fix a path $P=(x_0, \dots, x_m)$ (containing all the points of $\mathbf{K}$) and write $\RG(\mathbf{K})$ to ease the notation. 

In order to prove \ref{prop:of:rigidif:2}, note that the simplicial complex $\RG(\mathbf{K})$ is given by the union $\RG(\mathbf{K})= \mathbf{K} \cup \bigcup_{g \in G} \mathbf{B_g}$ and, for each $g \in G$, the path $g \cdot P$ is in both $\mathbf{K}$ and $\mathbf{B_g}$. Hence, the geometric realization of $\RG(\mathbf{K})$ can be described as the following pushout
     \begin{equation}
            \begin{tikzcd}
                 \vert \bigcup_{g \in G} g \cdot P \vert \arrow[r] \arrow[d] & \vert \bigcup_{g \in G} \mathbf{B_g} \vert  \arrow[d] \\
                \vert \mathbf{K} \vert \arrow[r] & \vert \RG(\mathbf{K}) \vert  \arrow[ul, phantom, "\text{\Huge$\ulcorner$}", pos=0]
            \end{tikzcd} \nonumber
        \end{equation}
where the map $\vert \bigcup_{g \in G} g \cdot P \vert \rightarrow \vert \mathbf{K} \vert$ is induced by the inclusion of each $g \cdot P$ and the map $\vert \mathbf{K} \vert \rightarrow \vert \RG(\mathbf{K}) \vert$ is the inclusion. 

Since $\vert \bigcup_{g \in G} \mathbf{B_g} \vert$ is contractible, collapsing it onto $\vert \bigcup_{g \in G} g \cdot P \vert$ in $\vert \RG(\mathbf{K}) \vert$, does not change the homotopy type \cite[Proposition 0.17]{hatcher} and yields $\vert \mathbf{K} \vert$ in the process. Thus, $\vert \mathbf{K} \vert$ and $\vert \RG(\mathbf{K}) \vert $ are homotopy equivalent. 

Let us now show \ref{prop:of:rigidif:1}. We will start by showing that an automorphism of $\RG(\mathbf{K})$ is determined by its restriction to $\mathbf{K}$ and that such restriction is an automorphism of $\mathbf{K}$. Then, we will see that the $G$-orbit and the $\RG(\mathbf{K})$-orbit of a tuple containing all the points of $\mathbf{K}$ coincide. The result will then follow by Lemma~\ref{lemma:g:orbits}.

Note that the points of $\mathbf{K}$ are in at least $\vert G \vert$ $(n+1)$-simplices. Since all points of $\RG(\mathbf{K}) \setminus \mathbf{K}$ are in at most one $(n+1)$-simplex, we conclude, using Lemma \ref{points:auto:simplices}, that $f(\mathbf{K})=\mathbf{K}$. Using Lemma \ref{lemma:glued:band}, we know that each piece $\mathbf{B_g}$ does not have automorphisms (when the points of $\mathbf{K}$ are mapped to $\mathbf{K}$). Therefore, by connectedness, the effect of the automorphisms of $\RG(\mathbf{K})$ outside $\mathbf{K}$ is just switching $\mathbf{B_g}$ pieces with each other. Such a switching is done according to what $f$ does on $\mathbf{K}$. 

In order to finish the proof, we show that the $G$-orbit and the $\aut(\RG(\mathbf{K}))$-orbit of a tuple containing all the vertices of $\mathbf{K}$ coincide. Consider the path 
$$Q=(x_0, \dots, x_j, \dots , x_m, w_{m}^e,u_{1,m}^e,\dots ,u_{m,m}^e)$$for $e \in G$ the identity. The image of this path must be of the same form, i.e., $$f(Q)=(x_{\sigma(0)},\dots x_{\sigma(j)},\dots, x_{\sigma(m)}, w_{m}^g,u_{1,m}^g,\dots ,u_{m,m}^g)$$ for some $g \in G$ and some permutation $\sigma$ of $\{0, \dots, m\}$. Furthermore, each $x_{\sigma(j)}$ must be the $j$-th element of one of the paths $g \cdot P$ for some $g \in G$ which means $x_{\sigma(j)}$ is in the $G$-orbit of $x_j$, hence the path $f(Q)$ is the following
$$f(Q)=(g \cdot x_0,\dots, g \cdot x_j,\dots, g \cdot x_m, w_{m}^g,u_{1,m}^g,\dots ,u_{m,m}^g) .$$
Consequently, the $G$-orbit of $(x_1, \dots,x_n)$ coincides with the $\aut(\RG(\mathbf{K}))$-orbit of $(x_1, \dots,x_n)$. Therefore, by Lemma~\ref{lemma:g:orbits}~, $\aut(\RG(\mathbf{K})) \cong G$.
\end{proof}

\begin{theorem} \label{simplicial:complex:thm}
Let $\mathbf{K}$ be a finite connected (abstract) simplicial complex and let $G$ be a subgroup of the full automorphism group $\aut(\mathbf{K})$ of $\mathbf{K}$. Then, there is a simplicial complex $\mathbf{L}$ such that:
\begin{enumerate}[label={\rm (\roman{*})}]
    \item $\aut(\mathbf{L}) \cong G$;
    \item $\vert \mathbf{L} \vert$ is homotopy equivalent to $\vert \mathbf{K} \vert$.
\end{enumerate}
\end{theorem}
\begin{proof}
For the case $G=\{1\}$, consider the set of vertices of $\mathbf{K}$, $V(\mathbf{K})=\{x_1, \dots, x_n\}$, and glue $\mathbf{W_{n+i}}$ in each point $x_i$ by identifying $v_2$ and $x_i$, for $1 \leq i \leq n$. If we relabel $v_2$ in $\mathbf{W_{n+i}}$ to $x_i$, then we can take $\mathbf{L}$ to be the union $\mathbf{L}= \mathbf{K} \cup \bigcup_{1\leq i \leq n} \mathbf{W_{n+i}}$. Then, the geometric realization of $\mathbf{L}$ consists of the following wedge sum
$$\vert \mathbf{L} \vert = \vert \mathbf{K} \vert \bigvee_{ \substack{x_i \sim v_2 \\ i=1}}^n \vert \mathbf{W_{n+i}} \vert$$
and since $\vert\mathbf{W_{n+i}} \vert$ is contractible for all $1 \leq i \leq n$, $\vert \mathbf{L} \vert$ is homotopy equivalent to $\vert \mathbf{K} \vert$. In order to check that the simplicial complex $\mathbf{L}$ is rigid, observe that $P=(t_1, \dots, t_{n+i+1})$ is the only path of $\mathbf{L}$ (with all points distinct) satisfying the following properties simultaneously:
\begin{itemize}
    \item $P$ has length $n+i+1$;
    \item $P$ starts at a point with exactly $2$ edges, one of which is incident to a point in a $2$-simplex;
    \item $P$ ends in a point with a single incident edge.
\end{itemize}
Therefore, by Lemma \ref{paths:auto:simplices} the image of $P$ under an automorphism of $\mathbf{L}$ must be $P$. This forces $x_n$ to be fixed by the automorphisms of $\mathbf{L}$. Successively applying this reasoning to the longest paths of $\mathbf{W_{n+i}} \subset \mathbf{L}$ for all $1 \leq i \leq n-1$ ensures that each $x_i$ for $1 \leq i \leq n$ is fixed by the automorphisms of $\mathbf{L}$. Since each $\mathbf{W_{n+i}}$ is rigid, we get that $\mathbf{L}$ has no non-trivial automorphisms.

Finally, for the case in which $G$ is not trivial, we take $\mathbf{L}$ to be the $G$-rigidification of $\mathbf{K}$ with respect to some path $P$ containing all the points of $\mathbf{K}$ (with possible repetitions). By Proposition \ref{prop:of:rigidif}, $\mathbf{L}=\RG(\mathbf{K};P)$ has $G$ as full automorphism group and the geometric realization $\vert \RG(\mathbf{K};P) \vert $ is homotopy equivalent to $\vert \mathbf{K} \vert$.
\end{proof}

Applying Theorem \ref{simplicial:complex:thm} to a simplicial complex $\mathbf{K}$ from Theorem \ref{simp:comp:covering}, we are able to realize a given action of a finite group on a finitely presentable group as the action of the full automorphism group of a simplicial complex on the fundamental group of the corresponding geometric realization. 

\begin{corollary} \label{simplicial:complex:cor}
    Let $G$ be a finite group acting on a finitely presentable group $M$. Then, there is a finite (abstract) simplicial complex $\mathbf{L}$ such that:
    \begin{enumerate}[label={\rm (\roman{*})}]
        \item\label{simplicial:complex:cor:1} $\pi_1(\vert \mathbf{L} \vert)\cong M$;
        \item\label{simplicial:complex:cor:2} $G \cong \aut(\mathbf{L})$;
        \item\label{simplicial:complex:cor:3} $G \curvearrowright M$ is equivalent to $ \aut(\mathbf{L}) \curvearrowright \pi_1(\vert \mathbf{L} \vert)$.
    \end{enumerate}
\end{corollary}
\begin{proof}
     Statements \ref{simplicial:complex:cor:1} and \ref{simplicial:complex:cor:2} are direct consequences of Theorem \ref{simplicial:complex:thm}. In order to obtain Statement \ref{simplicial:complex:cor:3}, recall that for a simplicial complex $\mathbf{K}$ from Theorem \ref{simp:comp:covering}, the action $G \curvearrowright M$ is equivalent to the action $G \curvearrowright \pi_1(\vert \mathbf{K} \vert)$ and that $G$ acted on $\mathbf{K}$ by automorphisms. Since the automorphisms of the $G$-rigidification of a (connected) simplicial complex $\mathbf{K}$, $\mathbf{L}=\RG(\mathbf{K})$, are determined by their restriction to $\mathbf{K}$, $\mathbf{L}$ inherits the same action. Moreover, now the group $G \leq \aut(\mathbf{K})$ is isomorphic to the full automorphism group of $\mathbf{L}=\RG(\mathbf{K})$. Hence, we obtain that the action $G \curvearrowright M$ is equivalent to the action $\aut(\mathbf{L}) \curvearrowright \pi_1(\vert \mathbf{L} \vert)$.
\end{proof}

\section{Approximation by a minimal finite space} \label{section:finite:spaces}
In this section, we want to translate the realizability result obtained in the previous section (Corollary \ref{simplicial:complex:cor}) from the category of abstract simplicial complexes $\asc$ to the category of topological spaces $\top$. There is a functor that is very useful for such a purpose. Recall that the category of finite $T_0$-spaces (with continuous maps) $\finite$ is isomorphic to the category of finite posets (with order-preserving maps) and consider the functor $$\mathcal{X} \colon \asc \longrightarrow \finite \subset \top$$ that takes a simplicial complex $\mathbf{K}$ to the poset $\mathcal{X}(\mathbf{K})$ of simplices of $\mathbf{K}$ (ordered by inclusion) and a simplicial map $f:\mathbf{K} \rightarrow \mathbf{L} $ to the order-preserving map $\mathcal{X}(f): \mathcal{X}(\mathbf{K}) \rightarrow \mathcal{X}(\mathbf{L}) $ defined by $\mathcal{X}(f)(\sigma)=f(\sigma)$, where $\sigma $ is a simplex of $\mathbf{K}$. Henceforth, we assume all finite spaces are $T_0$ (unless stated otherwise) and make no distinction between a finite space and its corresponding poset. 

Since automorphisms of simplicial complexes preserve inclusion of simplices and poset automorphisms preserve the order relation, we get that $\aut(\mathcal{X}(\mathbf{K})) \cong \aut(\mathbf{K})$. On the other hand, the finite space $\mathcal{X}(\mathbf{K})$ is weakly equivalent to $\vert \mathbf{K} \vert$ (see \cite[Theorem 1.4.12]{barmak}). Hence, we can translate Corollary \ref{simplicial:complex:cor} directly to finite spaces and obtain the following.

\begin{theorem} \label{thm:aut:finite:spaces}
Let $G$ be a finite group acting on a finitely presentable group $M$. Then there is a topological space $X$ such that:
\begin{enumerate}[label={\rm (\roman{*})}]
    \item\label{thm:aut:finite:spaces:1} $\pi_1(X)\cong M$;
    \item\label{thm:aut:finite:spaces:2} $G \cong \aut(X)$;
    \item\label{thm:aut:finite:spaces:3} $G \curvearrowright M$ is equivalent to $ \aut(X) \curvearrowright \pi_1(X)$.
\end{enumerate}
\end{theorem}
\begin{proof}
    All that is left to check is statement \ref{thm:aut:finite:spaces:3}. Fix $\mathbf{K}$ an abstract simplicial complex from Corollary~\ref{simplicial:complex:cor}. Let $\sd(\mathbf{K})$ denote the first barycentric subdivision of $\mathbf{K}$. Note that each vertex of $\sd(\mathbf{K})$ corresponds to a simplex of $\mathbf{K}$ and, hence, to a point of $\chi(\mathbf{K})$. Denote by $s: \vert\sd(\mathbf{K})\vert \rightarrow \vert \mathbf{K}\vert  $ the linear homeomorphism taking a simplex $\sigma$ of $\mathbf{K}$ to its barycenter. Each point $y$ of $\vert\sd(\mathbf{K)\vert}$ is a convex combination of the vertices $\{v_1 \dots,v_l\}$ of $\sd(\mathbf{K}) $. If we suppress the vertices $v_i$ for which $t_i=0$, we can write $y= \sum_{i=1}^rt_iv_i$, where $t_i >0$ for all $1 \leq i \leq r$. Now we define $\mu: \vert \sd(\mathbf{K}) \vert \rightarrow \chi(\mathbf{K})$ by assigning to $y$ the minimum of $\{v_1, \dots,v_r\}$ with respect to the order of $\chi(\mathbf{K})$.
    
    The composition $\mu \circ s^{-1}: \vert \mathbf{K} \vert \rightarrow \chi(\mathbf{K)}$ is a weak homotopy equivalence, as shown in \cite[Theorem 1.4.12]{barmak}. 

    The $G$-action on $\mathbf{K}$ is simplicial and for each $g \in G$ an $n$-simplex $\sigma=\{v_1, \dots,v_n\}$ is sent to the $n$-simplex $g\cdot \sigma=\{g\cdot v_1, \dots,g\cdot v_n\}$ by the linear homeomorphism taking $v_i$ to $g \cdot v_i$ for all $1 \leq i \leq n$. Therefore, the maps $\mu$ and $s$ are $G$-equivariant. This implies that $\mu \circ s^{-1}$ is $G$-equivariant, since inverting and composing equivariant maps yields equivariant maps. Consequently, if we denote by $\varphi$ the map taking an automorphism of $\mathbf{K}$ to the induced map on simplices and by $(\mu \circ s^{-1})_\bullet$ the map induced by $\mu \circ s^{-1}$ on fundamental groups, the following diagram commutes
    $$ \begin{tikzcd}
    \aut(\mathbf{K}) \times \pi_1(|\mathbf{K}|) \arrow[r,"\text{action}"] \arrow[d,"\varphi \times(\mu \circ s^{-1})_{\bullet}"] 
    & \pi_1(|\mathbf{K}|) \arrow[d,"(\mu \circ s^{-1})_{\bullet}"] \\
        \aut(\chi(\mathbf{K})) \times \pi_1(\chi(\mathbf{K})) \arrow[r,"\text{action}"] 
    & \pi_1(\chi(\mathbf{K})).
        \end{tikzcd}$$
    Since the action of $\aut(\mathbf{K})$ on $\pi_1(\vert \mathbf{K}\vert)$ was precisely the action of $G$ on $M$, we conclude the result.
\end{proof}
This is not exactly the statement of Theorem \ref{main:theorem}, since in Theorem \ref{thm:aut:finite:spaces}, we realize $G$ as the automorphism group of $X$ (that is, the group of homeomorphisms of $X$) and not as the group of self-homotopy equivalences $\E(X)$. In order to fix this, we need $X$ to be a \textit{minimal finite space}, since minimal finite spaces have the property that their automorphism group $\aut(X)$ is isomorphic to the group of self-homotopy equivalences $\E(X)$ \cite[Corollary 1.3.7]{barmak}. Unfortunately, applying the functor $\mathcal{X}$ to the $G$-rigidification $\RG(\mathbf{K})$ of a finite connected simplicial complex $\mathbf{K}$ does not generally yield a minimal finite space. Hence, in the remainder of this chapter we will show how to change the finite space $\mathcal{X}(\RG(\mathbf{K}))$ in a way that makes it minimal, but does not change the homotopy type nor the full automorphism group.

For the convenience of the reader, we first recall some concepts related to finite spaces. A finite space $X$ is usually represented by its \textit{Hasse diagram}, a directed graph with vertices corresponding to the points of $X$ and an edge between $x$ and $y$ if $x \lneq y$ and there is no $z \in X$ such that $x \lneq z \lneq y$. In that case, we say $y$ \textit{covers} $x$ and $x$ is \textit{covered} by $y$. A finite space is said to be \textit{minimal} if it has no \textit{beat points}, which are points covering, or covered by, a single point. A sequence of points $x_1 \lneq x_2 \lneq \dots \lneq x_n$ such that $x_i$ is covered by $x_{i+1}$ for all $1 \leq i \leq n-1$ is called a \textit{chain} of length $n$. We refer to \cite{barmak} for more details on finite spaces and record below specific finite spaces and two lemmas that are going to be key in the proof of Theorem \ref{main:theorem}.

\begin{example} \label{example:W2}
Let $W_2$ be the finite space whose Hasse diagram is in Figure \ref{DiagW2}, where the `rightmost' minimal point is labeled $a$ for later use. 
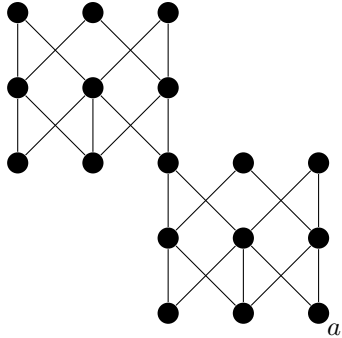
\begin{figure}[h]
 \begin{tikzpicture}
  [scale=1,auto=center,every node/.style={circle,fill=black,minimum size=8pt, inner sep=0pt}]

  \node (m1) at (2,0)  {};
  \node (m2) at (2,1)  {};
  \node (m3) at (2,2)  {};
  \node (m4) at (1,0)  {};
  \node (m5) at (1,1)  {};
  \node (m6) at (1,2)  {};
  \node (m7) at (0,0)  {};
  \node (m8) at (0,1)  {};
  \node (m9) at (0,2)  {};
  \node (m10) at (2,-1) {};
  \node (m11) at (2,-2) {};
  \node (m12) at (3,0) {};
  \node (m13) at (3,-1) {};
  \node (m14) at (3,-2) {};
  \node (m15) at (4,0) {};
  \node (m16) at (4,-1) {};
  \node (m17) at (4,-2) [label=below right:$a$] {};

  \foreach \from/\to in {m1/m2,m1/m5,m2/m3,m2/m6,m4/m5,m4/m2,m4/m8,m5/m7,m3/m5,m5/m9,m7/m8,m8/m9,m8/m6,m1/m10,m1/m13,m10/m11,m10/m14,m11/m13,m13/m14,m14/m16,m15/m16,m15/m13,m16/m12,m16/m17,m17/m13,m12/m10}
    \draw (\from) -- (\to);
\end{tikzpicture} 
\caption{Hasse diagram of $W_2=L_1 \vee L_1$}
\label{DiagW2}
\end{figure}

The space $W_2$ is interesting to us, because it is weakly contractible and rigid, that is, $\aut(W_2)=\{1\}$. Furthermore, the maximal chains of $W_2$ have length $5$. Observe that $W_2$ can be seen as the wedge sum of two copies of the same space, which we denote $L_1$, the space with $9$ points which is under the gluing point (and including it). 

We also define recursively the spaces $W_k$, for $k \geq 3$, by $W_k \vee L_1$. The space $W_k$ has similar properties to $W_2$: it is rigid and weakly contractible. However, the height of $W_k$ is $2k$, meaning we have weakly contractible rigid spaces with height as large as we want. See \cite[Example 6.1]{chocano1} and \cite[Section 3]{cgv:perm} for more details on $W_2$ and $W_k$, respectively.
\end{example}

The following simple lemma also appeared in \cite{cgv:perm}. We emphasize some different details that will appear in the proof of Theorem \ref{main:theorem}.
\begin{lemma} Let $X$ be a finite space and $f$ a homeomorphism of $X$. Then \label{facts_automorphism_finite_spaces}
    \begin{enumerate}[label={\rm (\roman{*})}]
    \item\label{fact1_automorphism_finite_spaces} $x_1 \leq x_2$ if and only if $f(x_1) \leq f(x_2)$, for all $x_1,x_2 \in X$. In particular, $f$ preserves beat points.
    \item\label{fact2_automorphism_finite_spaces} If $x_0 \lneq x_1 \lneq \dots \lneq x_n$ is a chain in $X$, then $f(x_0) \lneq f(x_1) \lneq \dots \lneq f(x_n) $ is also a chain of $X$.
    \end{enumerate} 
In particular, minimal and maximal points are preserved by automorphisms and the position of each point in a chain is also preserved by automorphisms.
\end{lemma}
\begin{proof}
    This is a consequence of $f$ and $f^{-1}$ being continuous, which is the same as monotone for finite spaces.
\end{proof}
The next lemma is a simple observation regarding the wedge sum of finite spaces. Here we follow the notation of \cite{barmak}: $\mathcal{K}(X)$ denotes the \textit{order complex} of the finite space $X$, which consists of the simplicial complex whose simplices are the non-empty chains of $X$. This construction provides a functor $\mathcal{K}: \finite \rightarrow \asc$ in the reverse direction of the functor $\mathcal{X}$ defined in the beginning of this section. The functor $\mathcal{K}$ has several interesting properties, namely for every finite space $X$, there is a weak homotopy equivalence $\vert \mathcal{K}(X) \vert \rightarrow X$. For more details on this construction, we refer to \cite[Section 1.4]{barmak}.
\begin{lemma}[\cite{barmak}, Proposition 4.3.10] \label{prop:wedge}
    Let $X$ and $Y$ be two (pointed) finite spaces. There exists a weak homotopy equivalence $\vert \mathcal{K}(X) \vert \vee \vert \mathcal{K}(Y) \vert \rightarrow X \vee Y$. In particular, if $Y$ is weakly contractible, then $X \vee Y$ and $X$ have the same weak homotopy type.
\end{lemma}
\begin{proof}
There are weak homotopy equivalences $$ \vert \mathcal{K}(X)\vert \vee \vert \mathcal{K}(Y) \vert \rightarrow \vert \mathcal{K}(X \vee Y) \vert \rightarrow X \vee Y.$$ Moreover, if $Y$ is weakly contractible, then the CW-complex $\vert \mathcal{K}(Y) \vert$ is contractible and therefore $\vert \mathcal{K}(X)\vert \vee \vert \mathcal{K}(Y) \vert \simeq \vert \mathcal{K}(X)\vert $. Since $X$ is weakly equivalent to $\vert \mathcal{K}(X) \vert$, we get that $X$ is weakly equivalent to $X \vee Y$.
\end{proof}

\begin{theorem} \label{main:theorem:text}
 Let $G$ be a finite group acting on a finitely presentable group $M$. Then there is a topological space $X$ such that:
    \begin{enumerate}[label={\rm (\roman{*})}]
        \item\label{main:theorem:text:1} $\E(X) \cong G$;
        \item\label{main:theorem:text:2} $\pi_1(X) \cong M$;
        \item\label{main:theorem:text:3} the action $G\curvearrowright M$ is equivalent to the canonical action $\E(X) \curvearrowright \pi_1(X)$.
    \end{enumerate}
\end{theorem}
\begin{proof}%[Proof of Theorem \ref{main:theorem}]
As in the proof of Theorem \ref{simplicial:complex:thm}, we separate two cases: when $G$ is the trivial group and when $G$ is not trivial. 

\textit{Case $G \neq \{1\}$:} The idea is to glue several copies of the space $W_l$ from Example~\ref{example:W2} (with $l$ sufficiently large) in appropriate points of a space obtained from Theorem~\ref{thm:aut:finite:spaces}. 

Recall that the spaces from Theorem \ref{thm:aut:finite:spaces} were obtained by applying the functor $\mathcal{X}$ to the $G$-rigidification (with respect to some path $P$) of a suitable simplicial complex $\mathbf{K}$. The next steps of this proof will rely heavily on these previous constructions, so let us fix an $n$-dimensional (connected) simplicial complex $\mathbf{K}$ such that:
\begin{enumerate}[label={\rm (\roman{*})}]
    \item $\pi_1(\vert \mathbf{K} \vert )\cong M$;
    \item $G$ acts simplicially on $\vert \mathbf{K} \vert$, hence $G \leq \aut(\mathbf{K})$;
    \item $G \curvearrowright M$ is equivalent to $ G \curvearrowright \pi_1(\vert \mathbf{K} \vert)$.
\end{enumerate}

Let $P$ be a path of $\mathbf{K}$ containing every point of $\mathbf{K}$ (with possible repetitions). Then apply to $\mathbf{K}$ the $G$-rigidification construction $\RG(\mathbf{K};P)$, denoted $\RG(\mathbf{K})$ henceforth. By Proposition \ref{prop:of:rigidif}, $\aut(\RG(\mathbf{K}))= G $ and the geometric realization $\vert \RG(\mathbf{K}) \vert $ is homotopy equivalent to $\vert \mathbf{K} \vert$. By the arguments described in the beginning of the current section, $\mathcal{X}(\RG(\mathbf{K}))$ is a finite space satisfying the properties of Theorem~\ref{thm:aut:finite:spaces}.

However, $\mathcal{X}(\RG(\mathbf{K}))$ is not a minimal finite space, since there are simplices of $\RG(\mathbf{K})$ which are contained in a single higher dimensional simplex. These simplices of $\RG(\mathbf{K})$ give rise to beat points in $\mathcal{X}(\RG(\mathbf{K}))$. In fact, the only points of $\mathcal{X}(\RG(\mathbf{K}))$ that are beat points are the ones coming from: $k$-simplices of $\mathbf{K} \subset \RG(\mathbf{K})$ contained in a single $(k+1)$-simplex for $1 \leq k \leq n-1$; $n$-simplices of $\RG(\mathbf{K}) \setminus \mathbf{K}$ which are subsimplices of a single $(n+1)$-simplex; and points/$0$-simplices which have a single incident edge. We glue a copy of the space $W_{l}$ from Example~\ref{example:W2} by the 'rightmost' minimal point (labeled $a$ in Figure~\ref{DiagW2}) on each of those points, where $2l >n+1$, let us say that $l= \lceil \frac{n+1}{2} \rceil+1$. The space $X$ we obtain after gluing a space $W_l$ in every beat point of $\mathcal{X}(\RG(\mathbf{K}))$ can be described as follows, where $W_l^x$ denotes the copy of $W_l$ glued on the point $x \in \mathcal{X}(\RG(\mathbf{K}))$.
$$X= \mathcal{X}(\RG(\mathbf{K})) \bigvee_{ \substack{x \sim a \\ x \text{ beat point}}} W_l^x $$
Since $W_l$ is a weakly contractible space, applying Lemma \ref{prop:wedge} several times implies that the space $X$ we obtained has the same weak homotopy type as $\mathcal{X}(\RG(\mathbf{K}))$ and therefore $\pi_1(\mathcal{X}(\RG(\mathbf{K}))) \cong \pi_1(X)$. Observe that we are implicitly changing the base point of $X$, so as to wedge in different points of $\mathcal{X}(\RG(\mathbf{K}))$; however, no issues arise from this change since we are working with path-connected spaces. On the other hand, no new automorphisms are added. 

Let us show that the group homomorphism $\varphi: \aut(X) \longrightarrow \aut(\mathcal{X}(\RG(\mathbf{K}))$ given by restriction $f \mapsto f|_{\mathcal{X}(\RG(\mathbf{K}))}$ is well-defined and bijective. Consider $f \in \aut(X)$. 

The key observation is that gluing points must be mapped to gluing points. Indeed, a beat point $x$ of $\mathcal{X}(\RG(\mathbf{K}))$ comes from some $k$-simplex of $\RG(\mathbf{K})$ contained in a single $(k+1)$-simplex. This means that in $X$, the point $x$ is in position $k+1$ of a chain of length $2l+k$ whose length is maximal, i.e., we can not add points to this chain to form a bigger chain. Furthermore, since $\hght(W_l)=l > n+1=\hght(X)$, the only chains of length $2l+k$ (and whose length is maximal) are those that start with $k+1$ points of $\mathcal{X}(\RG(\mathbf{K})) \subset X$ and end with $2l$ points of some $W_l \subset X$. Hence, by Lemma \ref{facts_automorphism_finite_spaces} the image of $x$ by an automorphism of $X$ must be another point that was a beat point in $\mathcal{X}(\RG(\mathbf{K}))$ coming from a $k$-simplex. In particular, this implies, by connectedness, that $f(\mathcal{X}(\RG(\mathbf{K}))=\mathcal{X}(\RG(\mathbf{K}))$ and that $f|_{\mathcal{X}(\RG(\mathbf{K})}$ is an automorphism of $\mathcal{X}(\RG(\mathbf{K}))$, as it is bijective and monotone. This shows the restriction map $\varphi$ is well-defined. Since $W_l$ is rigid and by connectedness $f(W_l^x)=W_l^{f(x)}$, the automorphism $f$ is determined by its restriction to $\mathcal{X}(\RG(\mathbf{K}))$. This proves injectivity of $\varphi$. To check surjectivity, take a $g \in \aut(\RG(\mathbf{K}))$ and consider the automorphism $\overline g: X \rightarrow X$ given by $g$ on $\mathcal{X}(\RG(\mathbf{K}))$ and by taking $W_l^x$ to $W_l^{g(x)}$, where $x$ is a beat point of $\mathcal{X}(\RG(\mathbf{K}))$, as the identity. 

Finally, since $W_l$ is a minimal finite space and the gluing was done on the beat points of $\mathcal{X}(\RG(\mathbf{K}))$, $X$ has no beat points, hence $X$ is a minimal finite space, and $\aut(X) \cong \E(X) \cong G$. Furthermore, $\pi_1(X) \cong M$ and the action of $G$ on $M$ is equivalent to the action of $\E(X)$ on $\pi_1(X)$, since $G$ acts by automorphisms as before.

\textit{Case $G= \{1\}$:} We follow a strategy similar to the previous case. Fix an $n$-dimensional connected simplicial complex $\mathbf{K}$ such that $\pi_1(\vert K \vert )\cong M$ and let $V(\mathbf{K})=\{x_1, \dots, x_m \}$ be the set of vertices of $\mathbf{K}$. Then take $\mathbf{L}$ to be the simplicial complex $$\mathbf{L}= \mathbf{K} \cup \bigcup_{1\leq i \leq m} \mathbf{W_{m+i}}$$ as in the proof of Theorem \ref{simplicial:complex:thm}, where $\mathbf{W_k}$ is the simplicial complex defined in Example~\ref{simp:comp:W_k,k+1} with $ m+1 \leq k \leq 2m$. By Theorem \ref{simplicial:complex:thm}, $\mathbf{L}$ is rigid ($\aut(\mathbf{L}) \cong G = \{1\}$) and $\pi_1(\vert \mathbf{L} \vert) \cong M$, as $\vert \mathbf{L}\vert$ is homotopy equivalent to $\vert \mathbf{K} \vert$. However, applying the functor $\mathcal{X}$ to $\mathbf{L}$ does not yield a minimal finite space, since $\mathbf{L}$ contains $0$-simplices which have a single edge, $1$-simplices which are in a single $2$-simplex and possibly more $k$-simplices of $\mathbf{K}$ contained in a single $(k+1)$-simplex for $1 \leq k \leq n-1$. Therefore, we will glue a space $W_l$ with $l= \lceil \frac{n+1}{2}\rceil$ on each of those beat points. We construct the following space, where $W_l^x$ denotes the copy of $W_l$ glued on the beat point $x \in \mathcal{X}(\mathbf{L})$.
$$X  = \mathcal{X}(\mathbf{L}) \bigvee_{ \substack{x \sim a \\ x \text{ beat point}}} W_l^x .$$

By Lemma \ref{prop:wedge}, $X$ is weakly homotopy equivalent to $\mathcal{X}(\mathbf{L}),$ which implies that \( \pi_1(X) \cong \pi_1(\mathcal{X}(\mathbf{L})) \). Therefore all that remains to show is that $X$ is rigid. Note that, once again, the gluing points are mapped to gluing points. By connectedness, this ensures that $f(\mathcal{X}(\mathbf{L})) = \mathcal{X}(\mathbf{L})$ and $f(W_l^x) = W_l^{f(x)}.$ The restriction $f|_{\mathcal{X}(\mathbf{L})}$ is also an automorphism. Since both $W_l$ and $\mathcal{X}(\mathbf{L})$ are rigid, it follows that $X$ is rigid as well. Consequently, we obtain a minimal finite space $X$ satisfying $\aut(X) \cong \E(X) \cong \{1\}$ and $\pi_1(X) \cong M.$
\end{proof}

We finish by showing that, as a consequence, of Theorem~\ref{main:theorem}, every action of a finite group $G$ on a finitely generated abelian group $M$ arises as the action of the group of self-homotopy equivalences $\E(X)$ of a space $X$ on one of the higher homotopy groups $\pi_k(X)$ for some $k \geq 2$. Before we prove this assertion, we recall that the \textit{non-Hausdorff join} $X \circledast Y$ of two finite spaces $X$ and $Y$ is the disjoint union of $X$ and $Y$ with partial order given by keeping the orderings within $X$ and $Y$ and adding the relations $x \leq y$ for all $x \in X$ and $y \in Y$. As described in the beginning of \cite[Section 2.7]{barmak}, we have the following sequence of maps
$$   \vert \mathcal{K}(X \circledast Y)\vert=\vert \mathcal{K}(X) \ast \mathcal{K}(Y)\vert \cong \vert \mathcal{K}(X) \vert \ast \vert \mathcal{K}(Y) \vert $$ 
where $\ast$ denotes both the join of simplicial complexes and of topological spaces. When $Y=S^0$ the above equation yields a weak equivalence between $X \circledast S^0$ and the suspension $\Sigma \vert \mathcal{K}(X)\vert$ of the CW-complex $\vert \mathcal{K}(X)\vert$
\begin{equation}
    \label{equation:suspension} X\circledast S^0 \simeq \vert \mathcal{K}(X \circledast S^0)\vert=\vert \mathcal{K}(X) \ast \mathcal{K}(S^0)\vert \cong \vert \mathcal{K}(X) \vert \ast \vert \mathcal{K}(S^0) \vert  \simeq \Sigma \vert \mathcal{K}(X)\vert. \end{equation}
\begin{corollary}\label{main:cor:text}
    Let $k \geq 2$ be a natural number, $G$ a finite group and $M$ a finitely generated $\Z G$-module. Then, there is a topological space $X$ such that:
        \begin{enumerate}[label={\rm (\roman{*})}]
            \item\label{main:cor:text:1} $\E(X) \cong G$;
            \item\label{main:cor:text:2} $\pi_k(X) \cong M$;
            \item\label{main:cor:text:3} the action $G\curvearrowright M$ is equivalent to the canonical action $\E(X) \curvearrowright \pi_k(X)$.
        \end{enumerate}
\end{corollary}
\begin{proof}
Let $Y$ be a space obtained from the construction in the proof of Theorem~\ref{main:theorem:text}. Then $Y$ is a minimal finite space and the following hold:
\begin{enumerate}[label={\rm (\roman{*})}]
            \item $\aut(Y) \cong \E(Y) \cong G$;
            \item $\pi_1(Y) \cong M$;
            \item the action $G\curvearrowright M$ is equivalent to the canonical action $\E(Y) \curvearrowright \pi_1(Y)$.
        \end{enumerate}
Observe that we can use Theorem~\ref{main:theorem:text}, since a finitely generated abelian group is also finitely presented. 

Now we consider the finite space $X$ obtained by the $(k-1)$-fold non-Hausdorff join of $Y$ with $W_2 \sqcup\{p\}$, where $W_2$ is the space from Example~\ref{example:W2} and $\sqcup$ denotes disjoint union. More precisely, $X$ is defined as follows
$$X = Y \circledast \underbrace{(W_2 \sqcup \{p\}) \circledast \dots \circledast(W_2 \sqcup\{p\}).}_{k-1 \text{ times}}$$
Since the non-Hausdorff join of minimal finite spaces is minimal \cite[Lemma~3.4 (i)]{cgv:perm}, the space $X$ is minimal. Furthermore, the space $W_2 \sqcup \{p\}$ is rigid, hence by \cite[Lemma~3.4 (iv)]{cgv:perm}, 
$$\aut(X)= \aut(Y) \times \aut(W_2 \sqcup \{p\})\times \dots   \times \aut(W_2 \sqcup \{p\}) \cong G \times \{1\} \times \dots \times \{1\}\cong G.$$
On the other hand, since $\vert \mathcal{K}(W_2 \sqcup\{p\})\vert$ is homotopy equivalent to $S^0$, we obtain that $X$ is weakly equivalent to $\Sigma^{k-1}\vert \mathcal{K}(Y)\vert$, by a similar argument to the one presented in (\ref{equation:suspension}). Thus, by Hurewicz's Theorem (since $\pi_1(X)$ is abelian) and the suspension isomorphism for homology we get
$$\pi_1(Y)\cong \pi_1(\vert \mathcal{K}(Y) \vert) \cong H_1(\vert \mathcal{K}(Y)\vert) \cong H_{k}(\Sigma^{k-1}\vert \mathcal{K}(Y)\vert) $$
and, since $\Sigma^{k-1} \vert\mathcal{K}(Y)\vert$ is $(k-1)$-connected, Hurewicz's Theorem yields
$$H_{k}(\Sigma^{k-1}\vert \mathcal{K}(Y)\vert) \cong \pi_k(\Sigma^{k-1}\vert \mathcal{K}(Y) \vert) \cong \pi_k(X).$$
Therefore, we obtain that $\pi_k(X) \cong \pi_1(Y) \cong M$. In order to see that the canonical action $\E(X) \curvearrowright \pi_k(X)$ is equivalent to $G \curvearrowright M$, observe that the maps involved in the above sequences of isomorphism are natural at the level of topological spaces. With the exception of the last map, these are either the suspension isomorphism for homology, the Hurewicz map or the composition of the functors $\vert \cdot \vert$ and $\mathcal{K}$. The last map is a composition of $\vert\mathcal{K}(\cdot)\vert$ with homeomorphisms and a map that collapses $\vert \mathcal{K}(W_2 \sqcup \{p\})\vert$ onto $S^0$ in the join $\vert \mathcal{K}(X) \vert\ast \vert \mathcal{K}(W_2 \sqcup \{p\})\vert$, which is also natural with respect to the first factor.
\end{proof}

\bibliographystyle{plain} % Choose a style (e.g., plain, alpha, ieeetr, etc.)
\bibliography{references}

\end{document}

%% file: symmetric_presentations.tex
We will think of elements in $F_n$ as words in an alphabet of $n$ characters, namely $\{x_1,\ldots , x_n\}$. Therefore, given $w\in F_n$ and a group $M,$ we can also think of $w$ as a function $w\colon M^n\to M$ where $w(m_1,\ldots, m_n)$ is given by evaluating the word $w$ in $M$ by replacing $x_i$ by $m_i$.

Let us define symmetric presentations relative to a group action, a slight variation of the concept of symmetric presentation in \cite[Section 1]{SymPrep}.

\begin{definition}
    Let $G$ be a finite group acting on a finitely presentable group $M$. A presentation 
    $$M=\langle x_1,\ldots, x_n\, |\, w_j(x_1,\ldots, x_n)=1, 1\leq j\leq r\rangle$$
    is said to be a \textit{$G$-symmetric presentation of $M$} if there exist permutation representations $\xi\colon G\to \Sigma_n$ and $\rho\colon G\to \Sigma_r$ such that for every $g\in G$ the following hold
    \begin{enumerate}[label={\rm (\roman{*})}]
    \item $g(x_i)=x_{\xi(g)(i)},$ for $1\leq i\leq n$
    \item $w_j(x_{\xi(g)(1)},\ldots, x_{\xi(g)(n)})=w_{\rho(g)(j)}(x_1,\ldots, x_n)$
    \end{enumerate}
\end{definition}

\begin{remark}
Observe that we are actually making a slight generalization of the concept of symmetric presentation in \cite{SymPrep}. In fact, given a $G$-symmetric presentation of $M$, namely
$$M=\langle x_1,\ldots, x_n\, |\, w_j(x_1,\ldots, x_n)=1, 1\leq j\leq r\rangle,$$
and taking $P=\xi(G)$, then $PM$ as defined in \cite[Section 1]{SymPrep} equals $M.$ Observe that the $G$-action may not be faithful hence $P=\xi(G)$ may be non-isomorphic to $G$.

On the other hand, given a permutation group $P\leq \Sigma_r,$ and a presentation of $M$ with $r$-generators, the group $PM$ defined in \cite[Section 1]{SymPrep} admits a $P$-action and the given presentation is indeed a $P$-symmetric presentation.
\end{remark}

\begin{lemma} \label{lemma:symm:pres}
    Let $G$ be a finite group acting on a finitely presentable group $M$. Then there exists a $G$-symmetric presentation of $M$.
\end{lemma}
\begin{proof}
    Let $e\in G$ denote the neutral element, and let $$M=\langle m_1,\ldots, m_n\, |\, w_j(m_1,\ldots, m_n)=e, 1\leq j\leq r\rangle$$
    be any finite presentation of $M.$ Since $G$ acts on the elements $m_i\in M,$ for any given $g\in G$ and $1\leq i\leq n$ there exists a (non-necessary unique) word $\widetilde{w}_{i,g}\in F_n$ such that $g(m_i)=\widetilde{w}_{i,g}(m_1,\ldots,m_n).$ 

    Consider now the group $\widetilde{M}$ given by the presentation
\begin{equation}\label{eq:sym_pres}
\widetilde{M} = \left\langle \!\!\!
\begin{array}{l|l} 
x_{1,g},\ldots,x_{n,g},\quad g\in G
&\!\!
\begin{array}{l}
w_j(x_{1,g},\ldots, x_{n,g}) = e,\quad \forall\, g \in G; \\[0.3em]
\widetilde{w}_{i,g}(x_{1,h},\ldots, x_{n,h}) = x_{i,hg} ,\quad \forall\, g,h \in G
\end{array}
\end{array}
\!\!\!\!\right\rangle 
\end{equation}
that is endowed with a $G$-action via automorphisms given by $g(x_{i,h})=x_{i,gh}$. This makes the presentation \eqref{eq:sym_pres} $G$-symmetric. 

Since $$x_{i,g}=x_{i,eg}=\widetilde{w}_{i,g}(x_{1,e},\ldots, x_{n,e}),$$ the group $\widetilde{M}$ is generated by the elements $x_{i,e},$ $i=1,\ldots,n,$ and it is easily seen to be isomorphic to $M$. Indeed, the map $f\colon M\to \widetilde{M}$ given by $f(m_i)=x_{i,e}$ is a $G$-equivariant isomorphism; given $g\in G$
\begin{align*}
     g(f(m_i))&=g(x_{i,e})=x_{i,ge}=x_{i,g},\\
    f(g(m_i)) &=f(\widetilde{w}_{i,g}(m_1,\ldots,m_n))=\widetilde{w}_{i,g}(x_{1,e},\ldots,x_{n,e})=x_{i,eg}=x_{i,g},
\end{align*}
thus $g(f(m_i)=f(g(m_i)).$  Therefore \eqref{eq:sym_pres} is a $G$-symmetric presentation of $M.$ 
\end{proof}

By \cite[Corollary 1.28]{hatcher}, we know there is a CW-complex $W$ whose fundamental group is $M$. The construction of $W$ depends on a chosen presentation of $M$ and goes as follows: there is a single $0$-cell; the $1$-cells correspond to the generators of the (chosen) presentation of $M$; and the $2$-cells correspond to the relations of such presentation, with the attaching maps being prescribed by the relations. Hence, if we take $M$ with the $G$-symmetric presentation (\ref{eq:sym_pres}) from the proof of Lemma \ref{lemma:symm:pres}, then $W$ inherits some additional properties: $G$ acts by cellular homeomorphisms on $W$ and the induced action of $G$ on the fundamental group of $W$, $\pi_1(W)\cong M$, is precisely the action of $G$ on $M$.